\documentclass[11pt,reqno]{amsart}

\usepackage[margin=1in]{geometry}
\usepackage[tt=false]{libertine}

\usepackage{amssymb,mathrsfs}

\usepackage[varbb]{newpxmath}
\let\savedbigtimes\bigtimes
\let\bigtimes\relax
\usepackage{mathabx} 
\let\bigtimes\savedbigtimes

\usepackage{graphicx}
\usepackage{enumerate}
\usepackage{bm}

\usepackage{bbm}

\usepackage{verbatim}
\usepackage{hyperref,color}
\usepackage[capitalize,nameinlink]{cleveref}
\usepackage[dvipsnames]{xcolor}
\hypersetup{
	colorlinks=true,
	pdfpagemode=UseNone,
    citecolor=OliveGreen,
    linkcolor=NavyBlue,
    urlcolor=black,
	pdfstartview=FitW
}
\usepackage{appendix}
\crefname{appsec}{Appendix}{Appendices}
\usepackage{tikz}

\theoremstyle{plain}
\newtheorem{theorem}{Theorem}[section]

\newtheorem{lemma}[theorem]{Lemma}

\theoremstyle{definition}
\newtheorem{definition}[theorem]{Definition}

\newtheorem*{assumption*}{Assumption}

\theoremstyle{remark}

\crefname{lemma}{Lemma}{Lemmas}
\crefname{theorem}{Theorem}{Theorems}
\crefname{definition}{Definition}{Definitions}
\crefname{fact}{Fact}{Facts}
\crefname{claim}{Claim}{Claims}
\crefname{proposition}{Proposition}{Propositions}

\newcommand{\E}{\mathbb{E}}

\renewcommand{\epsilon}{\varepsilon}

\newcommand{\Q}{\mathbb{Q}}

\newcommand{\QQ}{\mathbb{Q}}
\newcommand{\PP}{\mathbb{P}}
\newcommand{\EE}{\mathbb{E}}

\newcommand{\beq}{\begin{equation}}
\newcommand{\eeq}{\end{equation}}

\DeclareMathOperator{\supp}{supp}

\newcommand{\bA}{\mathbf{A}}

\newcommand{\bY}{\mathbf{Y}}
\newcommand{\bV}{\mathbf{V}}
\newcommand{\bB}{\mathbf{B}}

\newcommand{\cZ}{\mathcal{Z}}

\renewcommand{\emptyset}{\varnothing}

\begin{document}
	
\title{A second moment proof of the spread lemma}

\author[E.\ Mossel, J.\ Niles-Weed, N.\ Sun, I.\ Zadik]{Elchanan Mossel$^{\star\circ}$, Jonathan Niles-Weed$^\dagger$, Nike Sun$^\star$, and Ilias Zadik$^\star$}
\thanks{$^\star$Department of Mathematics, MIT;
$^\circ$MIT Institute for Data, Systems, and Society;
$^\dagger$Center for Data Science \& Courant Institute of Mathematical Sciences, NYU. Email: \texttt{\{elmos,nsun,izadik\}@mit.edu}; \texttt{jnw@cims.nyu.edu}}

\date{\today}

\begin{abstract} 
This note concerns a well-known result which we term the ``spread lemma,'' which establishes the existence (with high probability) of a desired structure in a random set. The spread lemma was central to two recent celebrated results: (a) the improved bounds of Alweiss, Lovett, Wu, and Zhang (2019) on the Erd\H{o}s-Rado sunflower conjecture; and (b) the proof of the fractional Kahn--Kalai conjecture by Frankston, Kahn, Narayanan and Park (2019).  While the lemma was first proved (and later refined) by delicate counting arguments, alternative proofs have also been given, via Shannon's noiseless coding theorem (Rao, 2019), and also via manipulations of Shannon entropy bounds (Tao, 2020). 

In this note we present a new proof of the spread lemma, that takes advantage of an explicit recasting of the proof in the language of Bayesian statistical inference. We show that from this viewpoint the proof proceeds in a straightforward and principled probabilistic manner, leading to a truncated second moment calculation which concludes the proof. The proof can also be viewed as a demonstration of the ``planting trick'' introduced by Achlioptas and Coga-Oghlan (2008) in the study of random constraint satisfaction problems. 
\end{abstract}

\maketitle

\section{Introduction}

 In this note, we present a new proof of a known result which we call the ``spread lemma.''  It was the key technical result behind the recent notable improvements (\cite{sunflower_annals}) in the sunflower conjecture posed by Erd\H{o}s and Rado (\cite{erdos1960intersection}) (see also (\cite{bell2021note})). Moreover, it has been also the key technical step behind the proof of the fractional Kahn--Kalai conjecture (\cite{fracKK_annals}), posed by Talagrand (\cite{MR2743011}) as a relaxation of the Kahn--Kalai conjecture (\cite{kahn2007thresholds}). Finally, we recently used it to obtain a proof of a modification of the second Kahn--Kalai conjecture (\cite{mossel2022second}). 

The spread lemma is proved in \cite{sunflower_annals},
and refined in \cite{fracKK_annals}, by delicate counting arguments. Alternative proofs have also been given by information-theoretic arguments: via the Shannon noiseless coding theorem (\cite{Rao_sunflower}), and also via manipulations of entropic quantities (\cite{Tao_sunflower}). 

In this work, we present a new proof of the spread lemma, inspired by ideas from Bayesian statistical inference. We give an explicit statistical interpretation of the setting, and show that this gives rise to a proof of the spread lemma in a principled and straightforward way. The argument can be viewed as an application of the ``planting trick'' introduced by
\cite{Dim_Amin} in the setting of random constraint satisfaction problems. The specific technique employed in this note is most closely related to recent applications of the planting trick in the context of \emph{``all-or-nothing'' phase transitions} in inference problems \cite{coja2022statistical}.

\subsection{The spread lemma} We first describe the spread property in detail. We work on a universe $X$ of $N$ elements. 
We now define what it means for a random subset to be \emph{spread}:
\begin{definition}\label{d:spread}
For $R>1$, we say that a random subset $\bA\sim \pi$ of $X$ is \emph{$R$-spread} if for every fixed $S \subseteq X$,  
	\begin{equation} \label{spread_cond}
	{\pi}(S \subseteq \bA) \leq \frac{1}{R^{|S|}}\,.
	\end{equation}
We also say that the probability measure $\pi$ is $R$-spread.
\end{definition}
Hence, a random subset of $X$ is spread, as long as it does not contain any fixed subset with (appropriately defined) significant probability. The definition can also be naturally understood as a set-theoretic analogue of a tail condition for real-valued random variables. Let $\bV$ denote the (independent of $\pi$) $p$-biased random subset of $X$ which includes each element independently with probability $p$, and let $\Q_p$ denote the law of $\bV$. The following is the spread lemma: 

\begin{theorem}[spread lemma]\label{thm:main}
Let $\mathscr{A}$ be a collection of subsets of $X$, with each subset of size at most $k$ with $k\ge2$. There exists a universal constant $C$ such that the following holds: if there exists an $R$-spread measure $\pi$ supported on $\mathscr{A}$, then for all $p \ge C(\log k)/R$ the (independent of $\pi$) $p$-biased random subset $\bV\sim\Q_p$ contains an element of $\mathscr{A}$ with probability at least $0.9$. \end{theorem}

As mentioned above, proofs of Theorem~\ref{thm:main} (or slight variants thereof) can be found in the prior works  \cite{sunflower_annals, fracKK_annals,Rao_sunflower, Tao_sunflower},  and the result has notable implications.  We refer for instance to \cite{Tao_sunflower} for an exposition of how the improved sunflower bounds \cite{sunflower_annals} can be deduced from Theorem~\ref{thm:main}. The proof of the fractional Kahn--Kalai conjecture, given Theorem~\ref{thm:main}, follows by a linear programming duality argument explained in \cite[Proposition 1.5]{fracKK_annals} (and first observed in \cite{MR2743011}).  In Section~\ref{s:proof} we present our proof of Theorem~\ref{thm:main}, and conclude with some discussions and comparison with prior work in Section~\ref{s:comments}. 

\section{Proof of Theorem~\ref{thm:main}}
\label{s:proof}

In this section we present our proof of Theorem~\ref{thm:main}. The core of the argument is the following intermediate result:
\begin{theorem}\label{thm:medium} 
Suppose $\bA\sim\pi$ is an $R$-spread random subset of $X$. Then there exists a coupling $\PP_p$ of $\bA\sim\pi$, $\bA'\sim\pi$, and $\bV\sim\Q_p$ under which  $\bA,\bV$ are independent,  and it holds that
	\beq\label{iter_exp_mod}
	\EE_{\PP_p} \bigg[\frac{|\bA' \setminus \bV|}{|\bA|}
	\mathbf{1}\{\bA\ne\emptyset \}\bigg]
	\leq \frac{7}{(pR)^{1/3}} \,.
	\eeq
 Moreover,  the marginal law of $\bA'$ is the same as that of $\bA$,  and therefore 
the random set $\bA'$ is also $R$-spread.
\end{theorem}

It is well known that (variants of) Theorem~\ref{thm:medium} imply Theorem~\ref{thm:main}; for completeness we give an argument that Theorem~\ref{thm:medium} implies Theorem~\ref{thm:main} in \S\ref{sec:implic}. We now proceed with the proof of Theorem~\ref{thm:medium}.

\subsection{The planted model}

We define the following ``planted model'': we first sample a random subset of $X$ according to $\pi$, denoted $\bA\sim\pi$, which we view as the (hidden) ``signal'' drawn from the prior $\pi$. We then sample \emph{independently} from $\bA$ the ``noise'' $\bV \sim \Q_p$. The ``observation'' is $\bY=\bA \cup \bV$. We then let $\bA'$ be sampled from the \emph{posterior distribution} of $\bA$ given $\bY$, that is, from the conditional law of $\bA$ given $\bY$. We let $\PP_p$ denote the resulting {joint} law of $(\bA,\bY,\bA')$; this will be the coupling mentioned in the statement of Theorem~\ref{thm:medium}. From Bayes's rule, the marginal laws of $\bA$ and $\bA'$ under $\PP_p$ must be the same, so $\bA'$ is also marginally distributed as $\pi$ under $\PP_p$. In particular, $\bA'$ is also $R$-spread and has the same support as $\bA$. Moreover, almost surely under $\PP_p$, we must have $\bA'\subseteq \bY=\bA \cup \bV$, therefore $\bA' \setminus \bV \subseteq \bA \cap \bA'.$ 

Abbreviate $\delta\equiv 1/(pR)^{1/3}$. Note that to prove Theorem~\ref{thm:medium} we may assume $7\delta<1$; otherwise, since $\bA' \setminus \bV \subseteq \bA \cap \bA'\subseteq\bA$, the result trivially holds. Moreover, since $\bA' \setminus \bV \subseteq \bA \cap \bA'$, it suffices to prove 
    \beq\label{iter_prob_mod} 
    \PP_p\Big(
    |\bA' \cap \bA | > \delta|\bA|\Big) 
    \le \frac{6}{(pR)^{1/3}} =  6\delta\,,
    \eeq
    from which it will follow that
    \beq\label{iter_exp_2_mod}
    \EE_{\PP_p} \bigg[\frac{|\bA' \cap \bA |}{|\bA|} \mathbf{1}\{\bA\ne\emptyset \}\bigg]
    \le \delta +  6\delta = 7\delta\,.
    \eeq

\subsection{The planting trick, and reduction to a second moment bound}
\label{ss:planting.trick}
We will compare the planted model (described above) with the ``null model'' where $\bY=\bV\sim\Q_p$. Define
	\begin{equation}\label{e:Z.Y}
	\cZ_{\bY}
	=\sum_{A'} \pi(A')
	\frac{\PP_p(\bY\,|\,\bA=A')}{\Q_p(\bY)}
	= \E_{\bA\sim\pi}\bigg[
	\frac{\PP_p(\bY\,|\,\bA)}{\Q_p(\bY)}\bigg]\,.
	\end{equation}
Let $\cZ_{\bY}(\bA,\delta)$ denote the contribution to the above sum from sets $A'$ with $|\bA\cap A'| > \delta |\bA|$. We use these quantities to re-express the bound \eqref{iter_prob_mod} in terms of the null model: 

\begin{lemma}\label{l:planted.to.null}
The bound \eqref{iter_prob_mod} is equivalent to
    \begin{equation}\label{eq:goal_1}
    \EE_{\PP_p}\bigg( 
    \frac{\cZ_{\bY}(\bA,\delta)}{\cZ_{\bY}}
    \bigg) \leq 6\delta\,. 
    \end{equation}
    
\begin{proof}
Under the planted model, the law of $\bY$ given $\bA$ can be written explicitly as
	\[
	\PP_p(\bY\,|\,\bA)
	=\frac{\mathbf{1}\{\bA\subseteq \bY\} \Q_p(\bY)}{p^{|\bA|}}\,.
	\]
The marginal law of $\bY$ under $\PP_p$ is therefore given by
	\begin{equation}\label{eq_plantedvsnull_def}
	\PP_p(\bY)
	= \sum_A\PP_p(\bA=A,\bY)
	=\Q_p(\bY) \sum_A
	\frac{\pi(A)\mathbf{1}\{A\subseteq \bY\} }{p^{|A|}}
	=\Q_p(\bY)  \cZ_{\bY}\,.
	\end{equation}
Then, by applying Bayes's rule, we obtain
	\[
	\PP_p(\bA'=A'\,|\,\bY)
	= \frac{\PP_p(\bA=A',\bY)}{\PP_p(\bY)}
	= \frac{1}{\cZ_{\bY}}
	\frac{\pi(A')\mathbf{1}\{A'\subseteq \bY\} }{p^{|A'|}}\,.
	\]
It follows that we can express
	\begin{equation}\label{e:cZ.delta}
	\PP_p\Big(|\bA'\cap\bA| > \delta|\bA| \,\Big|\,\bA,\bY\Big)
	=\sum_{A'}
\frac{ \mathbf{1}\{|A' \cap \bA| 
	> \delta|\bA| \} }{\cZ_{\bY}}
\frac{
 \pi(A')\mathbf{1}\{ A'\subseteq\bY\}}{p^{|A'|}} 
	=
	\frac{\cZ_{\bY}(\bA,\delta)}{\cZ_{\bY}}\,,
	\end{equation}
and the claim follows.
\end{proof}
\end{lemma}

The next step is our main application of the ``planting trick'': the observation is simply that since \eqref{eq_plantedvsnull_def} says $\cZ_{\bY}$ is the Radon--Nikodym derivative between the planted and null laws of $\bY$, it follows that it is unlikely to be small under the planted model. Thus, to show the desired bound \eqref{eq:goal_1}, it suffices to bound the expectation of $\cZ_{\bY}(\bA,\delta)$ under the planted model. The following lemma formalizes this:

\begin{lemma}\label{lem:planting_first}
The bound \eqref{eq:goal_1} from Lemma~\ref{l:planted.to.null} is implied if we can show
	\begin{equation}\label{eq:goal_2}
    \EE_{\PP_p}\Big(\cZ_{\bY}(\bA,\delta)
 	\Big)
    \leq 6\delta^2\,. 
    \end{equation}

\begin{proof}We saw in \eqref{eq_plantedvsnull_def}
that $\PP_p(\bY)=\Q_p(\bY)\cZ_{\bY}$, i.e., $\cZ_{\bY}$ is the Radon--Nikodym derivative between the planted and null laws of $\bY$. It follows that
	\[
	\PP_p( \cZ_{\bY} \leq \epsilon )   
	=\E_{\Q_p}\bigg( \frac{\PP_p(\bY)}{\Q_p(\bY)} \mathbf{1}\{
    	\cZ_{\bY}\le\epsilon \}\bigg) 
    	= \E_{\Q_p}\Big( \cZ_{\bY}
    	\mathbf{1}\{\cZ_{\bY}\le\epsilon\}\Big)
    \le\epsilon\,.\]
Setting $\epsilon=6^{1/2}\delta$ and combining with \eqref{eq:goal_2} gives
	\[\EE_{\PP_p}\bigg(
    \frac{\cZ_{\bY}(\bA,\delta)}{\cZ_{\bY}} \bigg)
	\leq
	\epsilon
	+\frac{\EE(
 	\cZ_{\bY}(\bA,\delta)
   	)}{\epsilon}
    \le 6\delta\,, \]
where we used that $\cZ_{\bY}(\bA,\delta) \le \cZ_{\bY}$. This shows that
\eqref{eq:goal_2} implies \eqref{eq:goal_1}, as claimed.
\end{proof}
\end{lemma}

\subsection{Truncated second moment calculation}\label{sec:trunc}

In this subsection we conclude the proof of Theorem~\ref{thm:medium} by showing that the bound \eqref{eq:goal_2} reduces to a (tractable) truncated second moment calculation: 

\begin{lemma}\label{l:truncated.second.mmt}
The bound \eqref{eq:goal_2} from Lemma~\ref{lem:planting_first} is equivalent to
\begin{equation}\label{eq:goal_3} 
\mathop{\E}_{\bA\sim\pi} \bigg[
\sum_{\ell > \delta|\bA|} 
\frac{\pi(|\bA_0 \cap \bA|=\ell \,|\, \bA)}{p^\ell}
\bigg]
\leq 6\delta^2\,,
\end{equation}
where $\bA_0\sim\pi$ is an independent copy of $\bA$, so that $(\bA,\bA_0)\sim\pi^{\otimes2}$. 

\begin{proof}Suppose under $\PP_p$ that $\bA_0$ is an independent copy of $\bA$, so that marginally $(\bA,\bA_0)\sim\pi^{\otimes2}$ as in the statement of the lemma.
From the expansion of $\cZ_{\bY}(\bA,\delta)$ implied by \eqref{e:cZ.delta}, the desired bound \eqref{eq:goal_2} can be rewritten as
    \begin{align*}
    \EE_{\PP_p}\bigg(
    \frac{\mathbf{1}\{|\bA_0 \cap \bA| > \delta  |\bA|,
    \bA_0\subseteq \bY\}}{p^{|\bA_0|}} 
     \bigg)
    \le 6\delta^2\,. 
    \end{align*}
Recall that $\bY=\bA \cup \bV, \bV \sim \QQ_p$. Thus, conditional on $\bA$ and $\bA_0$, we have $\bA_0\subseteq \bY$ with probability $p^{|\bA\setminus \bA_0|}$. It follows that the above is equivalent to
\begin{align*}
\mathop{\EE}_{(\bA_0, \bA) \sim \pi^{\otimes 2}} \bigg(
\frac{\mathbf{1}\{ |\bA_0 \cap \bA| > \delta|\bA|\} }
	{p^{|\bA_0 \cap \bA|} }
	  \bigg)
\le 6\delta^2\,. 
\end{align*} This in turn is equivalent to the stated bound
\eqref{eq:goal_3}.
\end{proof}
\end{lemma}

\begin{proof}[Proof of Theorem~\ref{thm:medium}]
As established in the preceding subsections (Lemmas~\ref{l:planted.to.null}--\ref{l:truncated.second.mmt}), it suffices to prove \eqref{eq:goal_3}. Conditional on $\bA$, we have
	\[
	\pi\Big(
		|\bA_0 \cap \bA|=\ell \,\Big|\, \bA  \Big)
	= \sum_{S \subseteq \bA, |S|=\ell} \pi (\bA_0 \cap \bA=S 
		\,\Big|\, \bA  \Big)
	\,.\]
Since $\bA_0\sim\pi$ is $R$-spread, it follows from Definition~\ref{d:spread} that 
	\[
	\pi\Big(
		\bA_0 \cap \bA = S \,\Big|\, \bA  \Big)
	\le \pi\Big(S\subseteq \bA_0 \,\Big|\, \bA  \Big)
	= \pi(S\subseteq \bA_0)
	\le \frac{1}{R^{|S|}}
	\]
for all $S\subseteq X$. Substituting this bound into the previous expression gives
	\[\pi\Big(
		|\bA_0 \cap \bA|=\ell \,\Big|\, \bA  \Big)
	\le \binom{|\bA|}{\ell} \frac{1}{R^{|S|}}
	\le \bigg(\frac{e |\bA| }{R\ell} \bigg)^\ell\,,
	\]
where the last step uses a standard bound on the binomial coefficient.

Recalling that $\delta=(pR)^{-1/3} \le 1/7$, we conclude that 
\begin{align*}
	&\sum_{\ell > \delta |\bA|} 
	\frac{\pi(|\bA_0 \cap \bA|=\ell  )}{p^\ell}
	\leq \sum_{\ell > \delta |\bA|} 
		\left(\frac{e |\bA| }{pR\ell }  \right)^{\ell} 
	\le \sum_{\ell > \delta |\bA|} 
		\left(\frac{e  }{pR\delta}  \right)^{\ell} 
	\\
&\qquad  
\leq \sum_{\ell >\delta|\bA|} \left(e \delta^{2} \right)^{\ell} 
\le 2e \delta^{2} \leq 6\delta^{2}\,,
\end{align*}
where we have used that $\ell$ must be a positive integer, and so $\ell\ge1$. This concludes the proof.
\end{proof}

\subsection{Proof of main theorem}
\label{sec:implic} 

We now show how to deduce
Theorem~\ref{thm:main} from Theorem~\ref{thm:medium}. 


\begin{proof}[Proof of Theorem~\ref{thm:main}]
First note that we may assume $\log k\ll R$, since otherwise  for a large enough constant $C$ we will have $C(\log k)/R\ge1$, in which case the theorem holds vacuously. Now let
	\beq\label{e:def.q}
	q = \frac{(700)^3}{R} \in(0,1)\,.
	\eeq
Let $(\bV_i)_{i\ge1}$ be i.i.d.\ samples from $\Q_q$. 
Let $\pi_1\equiv\pi$ be the law of $\bA_1\equiv\bA$.
Applying Theorem~\ref{thm:medium} gives a coupling of $\bA_1$ and $\bV_1$ with $\bB_1\equiv (\bA_1)'$ with the bound \eqref{iter_exp_mod}: writing 
$\bA_2\equiv \bB_1\setminus \bV_1$,
\eqref{iter_exp_mod} says
	\[\EE_{\PP_q} \bigg[\frac{|\bA_2|}{|\bA_1|}
	\mathbf{1}\{\bA_1\ne\emptyset\}\bigg]
	=\EE_{\PP_q} \bigg[\frac{|\bB_1\setminus\bV_1|}{|\bA_1|}
	\mathbf{1}\{\bA_1\ne\emptyset\}\bigg]
	\leq \frac{7}{(qR)^{1/4}}\,.
	\]
We then let $\pi_2$ be the law of $\bA_2$  conditional on $\bV_1$.
 We claim that $\pi_2$ is also $R$-spread, so that this procedure can be iterated to produce an overall coupling $\PP$ of $\bA_1$, $\bB_1=(\bA_1)'$, $\bA_2=\bB_1\setminus\bV_1$, $\bB_2=(\bA_2)'$, and so on, up to $\bA_{m+1}=\bB_m\setminus\bV_m$. Let $\pi_\ell$ denote the law of $\bA_\ell$
 conditional on $\bV_1,\ldots,\bV_{\ell-1}$.  If we suppose inductively that $\pi_{\ell-1}$ is $R$-spread, then we have
	\begin{align*}
	\pi_\ell(S\subseteq \bA_\ell)
	&=\PP\Big(S\subseteq \bB_{\ell-1}\setminus\bV_{\ell-1}
		\,\Big|\, (\bV_i)_{i \le \ell-1} \Big) \\
	&\le \PP\Big(S\subseteq \bA_{\ell-1}
		\,\Big|\, (\bV_i)_{i \le \ell-1} \Big)
	= \pi_{\ell-1}(S\subseteq \bA_{\ell-1}) \le \frac{1}{R^{|S|}}\,,
	\end{align*}
where  the intermediate inequality uses that $\bB_{\ell-1}\subseteq\bA_{\ell-1}\cup\bV_{\ell-1}$, and the last equality uses that $\bA_{\ell-1}$ and $\bV_{\ell-1}$ are independent.  This verifies our claim that all the $\pi_\ell$ are $R$-spread, which implies that the above procedure is well-defined, and leads to the iterated bound (by H\"older's inequality) 
	\begin{align*}
	\E_{\PP}\Big[ |\bA_{m+1}|^{1/m}\Big]
	& = \E_{\PP}
	\bigg[ \bigg(
	|\bA_1|
	\prod_{i=1}^m\bigg\{\frac{|\bA_{i+1}|}{|\bA_i|}
	\mathbf{1}\{\bA_i\ne\emptyset\}\bigg\}
	\bigg)^{1/m}
	\bigg]
	\\
	&
	\le \bigg( k
	\prod_{i=1}^m \E_{\PP}
	\bigg[\frac{|\bA_{i+1}|}{|\bA_i|}\mathbf{1}\{\bA_i\ne\emptyset \}\bigg]
	\bigg)^{1/m}
	\le 
	\frac{7k^{1/m}}{(qR)^{1/3}}
	\stackrel{\eqref{e:def.q}}{=} \frac{k^{1/m}}{100}\,.
	\end{align*}
In particular, as long as $m\ge \log k$, we have $\PP(\bA_{m+1}\ne\emptyset)
\le \E[|\bA_{m+1}|^{1/m}] \le 0.1$. 

Recall again that $\pi_{m+1}$ is the law of $\bA_{m+1}=\bB_m\setminus\bV_m$ conditional on $\bV_1,\ldots,\bV_m$. If $T_{m+1}\in\supp\pi_{m+1}$,  then there exists a sequence
	\[\Big(\bA_1(T_{m+1}),\bB_1(T_{m+1}),\ldots,
	\bA_m(T_{m+1}),\bB_m(T_{m+1})\Big)\,,\]
which occurs with positive probability under $\PP(\cdot\,|\,\bV_1,\ldots,\bV_m)$,  such that $T_{m+1}=\bB_m(T_{m+1})\setminus \bV_m$. Then, for $T_m \equiv \bB_m(T_{m+1})\in\supp\pi_m$,  the same reasoning gives a sequence
	\[\Big(\bA_1(T_m),\bB_1(T_m),\ldots,
	\bA_{m-1}(T_m),\bB_{m-1}(T_m)\Big)\,,\]
which occurs with positive probability under $\PP(\cdot\,|\,\bV_1,\ldots,\bV_{m-1})$,  such that
$T_m = \bB_{m-1}(T_m) \setminus \bV_{m-1}$. Let
$T_{m-1}\equiv \bB_{m-1}(T_m)$ and iterate to obtain
	\[
	T_{m+1}
	= T_m\setminus \bV_m
	= (T_{m-1}\setminus \bV_{m-1}) \setminus \bV_m
	= T_1 \bigg\backslash \bigg( \bigcup_{i=1}^m \bV_i\bigg)
	\]
for some $T_1\in\supp\pi$.  This shows that on the event $\bA_{m+1}=T_{m+1}=\emptyset$, the union of the $(\bV_i)_{i\le m}$ must cover some element $T_1\in\supp\pi$.

Let $\bV$ denote the union of the $(\bV_i)_{i\le m}$, and note that $\bV\sim\Q_p$ with $p=1-(1-q)^m$. We will take $m\ge \log k$ with
 $m\asymp\log k$, so that the assumptions imply
	\[
	mq \stackrel{\eqref{e:def.q}}{\asymp} \frac{\log k}{R} \ll 1\,.
	\]
It follows then that we can take
	\[
	p=1-(1-q)^m \asymp mq \asymp \frac{\log k}{R}\,,
	\]
which proves the claim.
\end{proof}

\section{Comments}\label{s:comments}

\subsection{Comparison with other proofs of Theorem \ref{thm:main}}

All previous proofs establish Theorem~\ref{thm:main}, by first proving some variant of Theorem~\ref{thm:medium}. The analogues of 
Theorem~\ref{thm:medium} in \cite{sunflower_annals,fracKK_annals}
are proved by careful counting arguments.\footnote{\cite{sunflower_annals} provided the first such argument, which was refined by \cite{fracKK_annals} for a stronger bound.} The alternate proof of \cite{Rao_sunflower} constructs an encoding scheme and derives a variant of Theorem~\ref{thm:medium} by the converse of Shannon’s noiseless coding theorem. Another related proof of \cite{Tao_sunflower} does not use counting or encoding arguments, but goes through non-trivial manipulations of Shannon entropy bounds\footnote{It was recently pointed out that the proof of \cite{Tao_sunflower} has a gap, which has been corrected in \cite{Hu_sunflower, Stoeckl_sunflower}.}.

The main contribution of our proof is the explicit recasting of the setting in the language of Bayesian statistical inference, which allows us to go back and forth between planted and null models in a useful way. Indeed, our proof is perhaps closest to the ones of \cite{sunflower_annals, fracKK_annals}, which may be regarded in some sense as first moment calculations under the planted model (although they are not stated as such), implemented by delicate counting arguments. By contrast, we transfer the planted model first moment condition to a null model second moment condition (\S\ref{ss:planting.trick}). Since $\Q_p$ is a simple product measure, the latter can be proved by a direct and simple expansion (\S\ref{sec:trunc}).  


\subsection{Comparison with direct second moment method}\label{sec:2ndMM}

In this final section, for expository purposes
we compare the calculation of \S\ref{sec:trunc} with a more naive second moment bound under the null model $\Q_p$. This would be perhaps the most standard probabilistic approach to attempt to directly prove Theorem~\ref{thm:main}. Under the null model where $\bY=\bV\sim\Q_p$, the quantity $\cZ_{\bY}$ from \eqref{e:Z.Y} and \eqref{eq_plantedvsnull_def} can be rewritten as
	\[
	\cZ_{\bV}
	=
	\sum_A \frac{\pi(A) \mathbf{1}\{A\subseteq \bV\}}{p^{|A|}}\,.
	\]
This is a particularly natural quantity if 
all sets in $\mathscr{A}$ have the same size $k$
(as is the case in many interesting applications),
and $\pi$ is uniform on $\mathscr{A}$: then $Mp^k\cZ_{\bV}$ is simply the number of $A\in\mathscr{A}$ that are contained in the random set $\bV$.  Abbreviating $\cZ\equiv\cZ_{\bV}$, the classical second moment method would seek to prove
	\begin{equation}\label{sec_mm}
	\frac{\EE_{\QQ_p}(\cZ^2)}{(\EE_{\QQ_p}\cZ)^2} 
	\leq \frac{10}{9}\,,
	\end{equation}
which would imply $ \Q_p(\cZ>0)  \geq 0.9$ by the Paley--Zygmund inequality. The desired result Theorem~\ref{thm:main} would follow, since $\cZ>0$ if and only if $\bV$ contains at least one $A\in\mathscr{A}$. To evaluate the left-hand side of \eqref{sec_mm}, note that $\EE_{\QQ_p}\cZ=1$, and
	\[
	\EE_{\QQ_p}(\cZ^2)
	=
	\sum_{A,A'} \frac{\pi(A)\pi(A') p^{|A\cup A'|}}{p^{|A|+|A'|}}
	=\sum_{A,A'} \frac{\pi(A)\pi(A') }{p^{|A \cap A'|}}
	\,.\]
Simplifying the last expression, we find that \eqref{sec_mm} is equivalent to
	\beq\label{2nd_MM}
	\sum_{\ell\ge0}
	\frac{\pi^{\otimes 2}\left( |\bA_0 \cap \bA|=\ell \right) }{p^{\ell}}
	\le \frac{10}{9}\,,
	\eeq
where $(\bA,\bA_0)\sim\pi^{\otimes2}$. Now note that the left-hand side of
\eqref{eq:goal_3} is exactly the left-hand side of
\eqref{2nd_MM} restricted to $\ell> \delta |\bA|$ --- this is the reason we referred to \eqref{eq:goal_3} as a ``truncated second moment bound.'' 

\subsection{The example of perfect matching}

We conclude this subsection by explaining that the unrestricted second moment bound \eqref{2nd_MM} does not always hold under the conditions of Theorem~\ref{thm:main}, which shows that the truncation in \eqref{eq:goal_3} is in fact necessary. For this, suppose $\mathscr{A}$ is the set of all perfect matchings of the complete graph $K_n$ (with $n$ even). In this case, straightforward calculations (following e.g.\ \cite[\S7, eq.~(28)]{fracKK_annals}) show that $R \asymp n$. However, we claim that the condition \eqref{2nd_MM} is \emph{not} satisfied for $p \asymp (\log k)/R  \asymp (\log n)/n$.  To this end, let $\bA_0,\bA\sim\pi^{\otimes2}$ be two independent uniformly perfect matchings, and note that the chance they share at least one edge can be bounded as
	\[
	\pi^{\otimes2}(\bA_0\cap\bA \ne \emptyset)
	\le \E|\bA_0\cap\bA|
	=\bigg( \frac{n}{2}\bigg)^2 \bigg/ \binom{n}{2}
	= \frac{1+o_n(1)}2\,.
	\]
Consequently, if we condition on $\bA_0=\{e_1,\ldots,e_{n/2}\}$, for each $i$ we have
	\[
	\pi\Big (\bA_0\cap\bA = \{e_i\} \,\Big|\, 
	\bA_0=\{e_1,\ldots,e_{n/2}\}\Big)
	\ge \frac{1-o_n(1)}{2n}\,.
	\]
(We have $(1+o_n(1))/n$ probability that $e_i\subseteq\bA$;
conditional on this, the probability that $\bA_0$ and $\bA$ share no additional edges is at least $(1-o_n(1))/2$ by the above observation.) Summing over $1\le i\le n/2$ gives
	\[
	\pi^{\otimes2}(|\bA_0\cap\bA|=1)\ge \frac{1-o_n(1)}{4}\,.
	\]
 It follows that the $\ell=1$ term in \eqref{2nd_MM}  diverges for $p\asymp (\log n)/n$. 

\section*{Acknowledgements} We thank Vishesh Jain for pointing out an error in an earlier version of the presented proof. We thank Youngtak Sohn, Terence Tao, and Lutz Warnke for additional helpful feedback. We also acknowledge the support of Simons-NSF grant DMS-2031883 (E.M., N.S., and I.Z.), the Vannevar Bush Faculty Fellowship ONR-N00014-20-1-2826 (E.M.\ and I.Z.), the Simons Investigator Award 622132 (E.M.), the Sloan Research Fellowship (J.N.W.), and NSF CAREER grant DMS-1940092 (N.S.).

\pagebreak

{\raggedright 
\bibliographystyle{alphaabbr}
\bibliography{pc}
}

\end{document}